\documentclass[10pt]{article}
\usepackage{t1enc}
\usepackage[latin1]{inputenc}
\usepackage[english]{babel}
\usepackage{amsmath,amsthm}
\usepackage{amsfonts}
\usepackage{latexsym}
\usepackage[dvips]{graphicx}
\usepackage{float}
\usepackage[natural]{xcolor}
\usepackage{algorithm}
\usepackage{algorithmic}
\usepackage{enumerate,enumitem}
\usepackage{multirow}
\usepackage{xcolor}
\usepackage[colorlinks,linkcolor=blue]{hyperref}
\usepackage{lineno}
\usepackage{setspace}
\usepackage{fullpage}
\usepackage[title]{appendix}
\usepackage{todonotes}
\usepackage{amssymb}
\usepackage{cite}

\usepackage{listings}
\newcommand{\eps}{\varepsilon}

\newcommand{\de}{\delta}
\newcommand{\be}{\beta}

\newcommand{\ga}{\gamma}
\newcommand{\la}{\lambda}

\def \N{\mathbb{N}}

\usepackage{bbm}
\newcounter{propcounter}
\usepackage{enumitem}
\theoremstyle{plain}
\newtheorem{thm}{Theorem}[section]
\newtheorem{theorem}[thm]{Theorem}
\newtheorem{conj}[thm]{Conjecture}
\newtheorem{lemma}[thm]{Lemma}
\newtheorem{corollary}[thm]{Corollary}
\newtheorem{proposition}[thm]{Proposition}

\theoremstyle{definition}

\newtheorem{question}[thm]{Question}

\newtheorem{remark}[thm]{Remark}

\newtheorem{definition}[thm]{Definition}
\newtheorem{claim}[thm]{Claim}
\newtheorem{fact}[thm]{Fact}

\newtheorem{example}[thm]{Example}

\newtheorem{defn-thm}[thm]{Definition-Theorem}
\numberwithin{equation}{section}

\newcommand{\btheorem}{\begin{theorem}}
\newcommand{\etheorem}{\end{theorem}}
\newcommand{\bconjecture}{\begin{conjecture}}
\newcommand{\econjecture}{\end{conjecture}}
\newcommand{\bproposition}{\begin{proposition}}
\newcommand{\eproposition}{\end{proposition}}
\newcommand{\bdefinition}{\begin{definition}}
\newcommand{\edefinition}{\end{definition}}
\newcommand{\bcorollary}{\begin{corollary}}
\newcommand{\ecorollary}{\end{corollary}}
\newcommand{\pr}{\begin{proof}}
\newcommand{\oof}{\end{proof}}
\newcommand{\bclaim}{\begin{claim}}
\newcommand{\eclaim}{\end{claim}}
\newcommand{\bquestion}{\begin{question}}
\newcommand{\equestion}{\end{question}}
\newcommand{\bfact}{\begin{fact}}
\newcommand{\efact}{\end{fact}}
\newcommand{\bremark}{\begin{remark}}
\newcommand{\eremark}{\end{remark}}
\newcommand{\eexample}{\end{example}}
\newcommand{\bexample}{\begin{example}}
\newcommand{\ma}{\end{lemma}}
\newcommand{\lem}{\begin{lemma}}

\newcommand{\mf}[1]{\mathbf{#1}}
\newcommand{\overf}[1]{\overleftarrow{\mathbf{#1}}}

\title{Compatible Powers of Hamilton Cycles in Dense Graphs}

\author{Xiaohan Cheng, Jie Hu, Donglei Yang\thanks{Data Science Institute, Shandong University, Shandong, China. Email: {\tt dlyang@sdu.edu.cn}.}}
\date{}

\makeatletter
\def\th@plain{%
  \upshape 
}
\makeatother

\makeatletter
\renewenvironment{proof}[1][\proofname]{\par
  \pushQED{\qed}%
  \normalfont \topsep6\p@\@plus6\p@\relax
  \trivlist
  \item[\hskip\labelsep
        \bfseries
    #1\@addpunct{.}]\ignorespaces
}{%
  \popQED\endtrivlist\@endpefalse
}
\makeatother

\begin{document}
\baselineskip 0.56cm
\maketitle
\begin{abstract}
\vspace{3mm}
Motivated by the concept of transition system investigated by Kotzig in 1968, Krivelevich, Lee and Sudakov proposed a more general notion of incompatibility system to formulate the robustness of Hamiltonicity of Dirac graphs.
Given a graph $G=(V,E)$, an {\em incompatibility system} $\mathcal{F}$ over $G$ is a family $\mathcal{F}=\{F_v\}_{v\in V}$ such that for every $v\in V$, $F_v$ is a family of edge pairs in $\{\{e,e'\}: e\ne e'\in E, e\cap e'=\{v\}\}$.
An incompatibility system $\mathcal{F}$ is \emph{$\Delta$-bounded} if for every vertex $v$ and every edge $e$ incident with $v$, there are at most $\Delta$ pairs in $F_v$ containing $e$.
A subgraph $H$ of $G$ is \emph{compatible} (with respect to $\mathcal{F}$) if every pair of adjacent edges $e,e'$ of $H$ satisfies $\{e,e'\} \notin F_v$, where $v=e\cap e'$.
Krivelevich, Lee and Sudakov proved that there is an universal constant $\mu>0$ such that
for every $\mu n$-bounded incompatibility system $\mathcal{F}$ over a Dirac graph,
there exists a compatible Hamilton cycle, which resolves a conjecture of
H\"{a}ggkvist from 1988. We study high powers of Hamilton cycles in this context and show that for every $\gamma>0$ and $k\in\N$, there exists a constant $\mu>0$ such that for sufficiently large $n\in\mathbb{N}$ and every $\mu n$-bounded incompatibility system over an $n$-vertex graph $G$ with $\delta(G)\ge(\frac{k}{k+1}+\gamma)n$, there exists a compatible $k$-th power of a Hamilton cycle in $G$. Moreover, we give a construction which has minimum degree $\frac{k}{k+1}n+\Omega(n)$ and contains no compatible $k$-th power of a Hamilton cycle.

~


\end{abstract}

\section{Introduction}
The classical theorem of Dirac \cite{dirac1952some} asserts that every graph of order $n\ge 3$ and minimum
degree at least $\frac{n}{2}$ contains a \emph{Hamilton cycle}, that is, a cycle passing through each vertex in the graph. Hamilton cycle is a very important
and extensively studied notion in graph theory.
Also, Dirac's theorem is a cornerstone result in extremal graph theory, and it has been generalized in several directions (see e.g.~\cite{cuckler2009hamiltonian,ferber2017counting,KLS01,kuhn2013optimal,sarkozy2003number}). One fruitful area is to establish the existence, under certain (minimum) degree conditions, of more general spanning graphs other than a Hamilton cycle. A remarkable direction is P\'osa and Seymour's conjecture on the existence of powers of Hamilton cycles. For $k\in\N$, the \emph{$k$-th power} of a Hamilton cycle is defined
as a graph on the same vertex set whose edges join distinct vertices at distance at most $k$ in the Hamilton cycle.
\begin{conj}\cite{posa60,seymour1973problem}\label{PS}
Let $G$ be a graph on $n$ vertices. If $\de(G)\ge \frac{k}{k+1}n$, then $G$ contains the $k$-th power of a Hamilton cycle.
\end{conj}
After two decades and several
papers on this question, Koml\'{o}s, S\'{a}rk\"{o}zy and Szemer\'{e}di \cite{Komlos1998On} resolved this conjecture for large $n$. 

In this paper, we are interested in a ``local resilience'' version of Conjecture~\ref{PS} under the incompatibility system, suggested by Krivelevich, Lee and Sudakov \cite{KLS01}.
\begin{definition}\label{system} (Incompatibility system). Let $G=(V,E)$ be a graph. An \emph{incompatibility system} $\mathcal{F}$ over $G$ is a family $\mathcal{F} = \{F_v\}_{v\in
V}$ such
that for every $v\in V$,
$F_v$ is a family of edge pairs in $\{\{e, e'\}:
e\cap e' = \{v\}\}$.
\begin{enumerate}
\item [(1)] For every two edges $e$ and $e'$, if there exists some vertex $v$ such that $\{e, e'\}\in F_v$,
then we say
that $e$ and $e'$ are
\emph{incompatible} at $v$. Otherwise, they are \emph{compatible}. A subgraph $H\subseteq G$ is \emph{compatible}
if all
pairs of adjacent edges are
compatible.
\item [(2)] For a positive integer $\Delta$, an incompatibility system $\mathcal{F}$ is
\emph{$\Delta$-bounded} if for
any vertex $v$ and every edge
$e$ incident with $v$, there are at most $\Delta$ other edges incident with $v$ that are incompatible with
$e$.
\item [(3)] Given constants $\mu,\delta>0$ and $n\in \N$, an\emph{ $(n, \delta, \mu)$-incompatibility system}
$(G, \mathcal{F})$
consists of an $n$-vertex
graph $G$ with $\delta(G)\geq \delta n$ and a $\mu n$-bounded incompatibility system $\mathcal{F}$ over
$G$.
\end{enumerate}
\end{definition}
The definition of incompatibility system is motivated by two concepts in graph theory.
First, it generalizes \emph{transition systems} introduced by Kotzig \cite{Kotzig1968} in 1968, where a transition system is a $1$-bounded
incompatibility system. Note that compatible subgraphs also generalize
the concept of properly colored subgraphs in edge-colored graphs, where two adjacent edges are incompatible if only they have the same color.

Krivelevich, Lee and Sudakov~\cite{KLS} first studied compatible Hamilton cycles in Dirac graphs and resolved in a very strong form, a conjecture of H\"aggkvist from 1988 (see Conjecture 8.40 in \cite{Bondy}).

\begin{theorem}\cite{KLS}
There exists a constant $\mu>0$ such that for large enough $n$, every $(n, \frac{1}{2}, \mu)$-incompatibility system contains a compatible Hamilton cycle.
\end{theorem}
They further studied compatible Hamilton cycles in random graphs in \cite{krivelevich2016compatible} and proved that there is a constant $\mu > 0$ such that if $p\gg \frac{\log n}{n}$, then $G=G(n, p)$ a.a.s.\footnote{Given a non-trivial monotone graph property $\mathcal P$, we say $G(n, p)$ has property $\mathcal{P}$ \emph{asymptotically almost surely} (or \emph{a.a.s.} in short) if the
probability that $G(n, p)$ has $\mathcal{P}$ tends to $1$ as $n$ tends to infinity.} has the following property: every $\mu np$-bounded incompatibility system defined over $G$ contains a compatible Hamilton cycle, which strengthen the result about Hamilton cycles in $G(n,p)$ without restrictions of incompatibility systems. The notion of incompatibility system appears to provide a new and interesting take on the robustness of graph properties. We refer the reader to a survey of Sudakov~\cite{sudakov2017robustness} where various measures of robustness and relevant results are collected.

We further explore the degree conditions forcing the existence of compatible high powers of Hamilton cycles and prove the following result.
\begin{theorem}[Main Theorem]\label{main theorem} For every $\gamma > 0$ and $k\in \N$ with $k\ge 2$, there exists a
constant $\mu >
0$ such that for
sufficiently large $n$, every
$(n, \frac{k}{k+1} + \gamma, \mu)$-incompatibility system $(G, \mathcal{F})$ contains the compatible $k$-th power of a Hamilton cycle.
\end{theorem}
In particular, the term $\gamma n$ in the minimum degree condition cannot be omitted by the constructions obtained in our previous work~\cite{hu2022graph}. We will give a simple construction for completeness.

\subsection{A space barrier}
We shall give an $(n, \frac{k}{k+1} + \frac{\mu}{2}, \mu)$-incompatibility system $(G, \mathcal{F})$ that contains no compatible $k$-th power of a Hamilton cycle.
Let $0<\mu<\frac{1}{2(k+1)},n\in (k+1)\N$ and $G_0$ be an $n$-vertex complete $(k+1)$-partite graph with parts $V_1,V_2,\ldots,V_{k+1}$ satisfying $|V_1|=\frac{n}{k+1}+1, |V_2|=\frac{n}{k+1}-1$ and $|V_i|=\frac{n}{k+1}$ for every $i\in\{3,\ldots,k+1\}$. Inside every part $V_i$ of $G_0$, we add a bipartite spanning subgraph with minimum degree at least $\frac{\mu n}{2}+1$ and maximum degree at most $\mu n$. Denote by $G$ the resulting graph. Hence, $\delta(G)\ge \left(1-\frac{1}{k+1}+\frac{\mu}{2}\right)n$ and for every $i\in[k+1]$, $G[V_i]$ is a triangle-free graph with $\delta(G[V_i])\ge \frac{\mu n}{2}+1$ and $\Delta(G[V_i])\le \mu n$.
Now we define an incompatibility system $\mathcal{F}$ over $G$. For every two different parts $V_i,V_j$ of $G$, let $v$ be any vertex in $V_i$ and $u,w$ be any two different vertices in $V_j$. If $uw$ is an edge in $G[V_j]$, then let $vu$ and $vw$ be incompatible at $v$. Since $\Delta(G[V_j])\le \mu n$, the resulting system $\mathcal{F}$ is $\mu n$-bounded. Furthermore, suppose for contradiction that $(G,\mathcal{F})$ contains a compatible $k$-th power of a Hamilton cycle. As $n\in (k+1)\N$, $(G,\mathcal{F})$ also contains a compatible $K_{k+1}$-factor, say $\mathcal{K}$. Then since $|V_1|=\frac{n}{k+1}+1=|\mathcal{K}|+1$ and each $G[V_i]$ is a triangle-free graph, by Pigeonhole Principle, there exists a compatible copy $K$ of $K_{k+1}$ in $\mathcal{K}$ which intersects $V_1$ on exactly two vertices, say $u_1$ and $w_1$. As $k\ge2$, $K$ also intersects another part except $V_1$, say $V_j$ for some $j\neq 1$, and choose $v_1\in V(K)\cap V_j$. Then $u_1v_1w_1$ is a compatible triangle, a contradiction.

The paper is organised as follows. In Section \ref{notation}, we set up some basic notation
and crucial lemmas. Then we present the proof of Theorem \ref{main theorem} in Section \ref{Pf of main}. Sections \ref{sectionabsorbing}, \ref{sectioncovering} and \ref{sectionconnecting} are devoted to proving Lemmas \ref{absorbing}, \ref{covering} and \ref{connecting1}, respectively.

\section{Notation and preliminaries} \label{notation}
Let $P_n$ be a path of order $n$. We use $P^k_{n}$ to denote the $k$-th power of $P_n$, where we often call $P_n$ the \emph{base} path. Given a copy of $P_{n}$, say $v_1v_2\ldots v_n$, we call the $k$-tuples $(v_k,\ldots,v_1)$ and $(v_{n-k+1},\ldots,v_{n})$ the \emph{ends} of $P_{n}$. More often, given a $k$-tuple $\mf{e}=(u_1,u_2,\ldots,u_k)$, we write $\overf{e}:=(u_k,u_{k-1},\ldots,u_1)$.

\begin{definition}\label{mate}  Let $G$ be an $n$-vertex graph and $\mathcal{F}$ be an
incompatibility system
over $G$. For every $k$-tuple $\mf{e}=(u_1,u_2,\ldots,u_k)$ such that $\{u_1,u_2,\ldots,u_k\}$ induces a compatible copy of $K_k$, we say $\mf{f}=(v_1,v_2,\ldots,v_k)$ is a \emph{mate} of $\mf{e}$ if $u_1\ldots u_{k}v_1\ldots v_{k}$ forms a compatible copy of $P^k_{2k}$ in $(G,\mathcal{F})$.
Denote by $M(\mf{e})$ the number of mates of $\mf{e}$ in $(G,\mathcal{F})$.
\end{definition}

\subsection{Absorption Strategy and main tools}
Our proof uses the absorption method, pioneered by the work of R\"odl, Ruci\'nski and Szemer\'edi~\cite{RRS} on perfect matchings in hypergraphs and here we shall follow the strategy in the work of Levitt, S\'{a}rk\"{o}zy and Szemer\'{e}di~\cite{levitt2010avoid} used for embedding powers of Hamilton cycles. The first major task is to define and find an absorbing structure in the host graph which can ``absorb'' any (small) set of left-over vertices. Here we adopt a weaker absorbing strategy in which we only need to absorb subsets of a fixed vertex set and a similar idea has previously appeared in a recent work of Chang--Han--Thoma~\cite{chang}. We will see that this weaker version is easier to handle. In particular, we will introduce a slightly stronger notion of absorbers to aid our proof.
\begin{definition}[$\beta$-absorber]\label{robust} Let $k\in \mathbb{N}$, and $G$ be an $n$-vertex graph
and
$\mathcal{F}$ is an
incompatibility system over $G$. For every $v\subseteq V$ and $\beta>0$, we say that a compatible copy $A$ of $P_{2k}^k$ is an
\emph{absorber }for $v$ if
$V(A)\cup\{v\}$ induces a compatible copy of
$P^k_{2k+1}$ which shares the same ends with $A$. In addition, if both ends have at
least $\beta n^{k}$ mates in $G$, then we call $A$ a $\beta$-\emph{absorber}.
\end{definition}
The first task is to ensure that every vertex has many $\be$-absorbers as above.
\begin{lemma}\label{absorbing}
Let $k\in \mathbb{N}$. For any $\gamma > 0$, there exist
$\beta_1,
\beta_2,\mu> 0$ such that for every sufficiently large $n$, if $(G, \mathcal{F})$ is an
$(n,\frac{k}{k+1} + \gamma, \mu)$-incompatibility system, then every vertex of $G$ has at least $\beta_1n$
vertex-disjoint
$\beta_2$-absorbers.
\end{lemma}
In \cite{levitt2010avoid}, a key step is to build a $k$-th power of a short path connecting every two fixed copies of $K_k$. We are attempting this approach which boils down to building compatible $k$-th power of a path for every two fixed compatible copies of $K_k$.
However, in our context, it is still unclear to us whether any compatible copy of $K_k$ has a mate. The worst-case scenario would be that an end of a compatible $k$-th power of a path can not be extended to longer ones and the bulk of
the work in our paper is to overcome this.
The second major work to cover almost all vertices using a constant number of compatible $k$-th power of paths in a robust manner that every end of the paths can still be `extended' further.
\begin{lemma}[Almost cover]\label{covering}
Let $k\in \N$. For any $\gamma, \tau > 0$, there exists
$\mu,\lambda,\beta > 0, $ such
that for every sufficiently
large $n$, if $(G, \mathcal{F})$ is an $(n,\frac{k}{k+1} + \gamma, \mu)$-incompatibility system, then
there exists a family of vertex-disjoint compatible $k$-th power of paths each of length at least $\lambda n$ covering all but at most $\tau n$
vertices of $G$, and every end of them has at least $\beta n^k$ mates.
\end{lemma}

 The following result is used for connecting two ends (compatible copies of $K_r$) via a compatible $k$-th power of a path provided that both of them have polynomially many mates.
\begin{lemma}[Connecting ends]\label{connecting1}
For any $\beta, \gamma > 0$, there exist $\mu>0$ and $L=L(\ga)\in \N$
such that for every sufficiently large $n$. Let $(G, \mathcal{F})$ be an $(n,\frac{k}{k+1} + \gamma, \mu)$-incompatibility system, and $W\subseteq V(G)$ with $|W|<\min\{\frac{\gamma}{2}n,\frac{\beta}{2}n\}$. Then for every two disjoint $k$-tuples of vertices $\mf{e}_1,\mf{e}_2$ each with a family $\mathcal{M}_i$ of at least $\beta n^{k}$ mates ($i\in[2]$), there exists a compatible $k$-th power of a path $Q$ of length at most $L$, which has $\overf{e}_1,\overf{e}_2$ as ends and all other vertices in $G-W$.

\end{lemma}
As such, Lemma~\ref{connecting1} allows us to connect two compatible $k$-th power of paths into a longer one.
\begin{lemma}[Connecting paths]\label{connecting} For any $\beta, \gamma>0$, there exist $\mu>0$ and $L=L(\ga)\in \N$
such that the following holds
for sufficiently large
$n$. Let $(G, \mathcal{F})$ be an $(n,\frac{k}{k+1} + \gamma, \mu)$-incompatibility system, and $W\subseteq
V(G)$
with $|W|<\min\{\frac{\gamma}{4}n,\frac{\beta}{4}n\}$. Suppose $G$ has two vertices $u,v$ and vertex-disjoint absorbers $A_u,A_v$ on base paths $P_1$ and $P_2$, respectively, such that each $P_i$ has an end $\mf{e}_i$ with $M(\mf{e}_i)\ge \beta
n^{k}, i=1,2$. Write \[P_1=a_1a_2\ldots a_ku_1u_2\ldots u_k, ~\mf{e}_1=(u_1,\ldots,u_k),\]
\[P_2=b_1b_2\ldots b_kv_1v_2\ldots v_k,~\mf{e}_2=(v_1,\ldots,v_k).\]Then there exists
a compatible $k$-th power of a path $Q$ of length at most $L$ in $G-W$, such that

  \stepcounter{propcounter}
\begin{enumerate}[label = ({\bfseries \Alph{propcounter}})]
       \item\label{abs} the $k$-th power of the paths $a_1a_2\ldots a_k u_1u_2\ldots u_k\sim Q\sim v_k\ldots v_1b_k\ldots b_2b_1$ and \[a_1a_2\ldots a_k \sim u\sim u_1u_2\ldots u_k\sim Q\sim v_k\ldots v_1\sim v\sim b_k\ldots b_2b_1\] are all compatible in $(G, \mathcal{F})$.
\end{enumerate}

\end{lemma}
\begin{proof}
For any $\beta, \gamma > 0$, we choose \[\frac{1}{n}\ll\mu\ll\beta,\gamma,\frac{1}{k}~\text{and additionally}~\frac{1}{L}\ll \ga .\]
We fix
$(G, \mathcal{F})$ to be an $(n,\frac{k}{k+1} + \gamma, \mu)$-incompatibility system, and $W\subseteq
V(G)$
with $|W|<\min\{\frac{\beta}{4}n,\frac{\gamma}{4}n\}$, and $P_1,P_2,\mf{e}_1,\mf{e}_2$ given as above such that $M(\mf{e}_i)\ge \beta
n^{k}, i=1,2$.
We say a mate $\mf{f}=(f_1,\ldots,f_k)$ of $\mf{e}_1$ (similarly for $\mf{e}_2$) is \emph{good} for $P_1$ if
the $k$-th power of the paths $a_1a_2\ldots a_ku_1u_2\ldots u_k\sim f_1f_2\ldots f_k$ and $a_1a_2\ldots a_k\sim u\sim u_1u_2\ldots u_k\sim f_1f_2\ldots f_k$ are compatible, and otherwise it is \emph{bad}. 
Thus if $\mf{f}$ is bad for $P_1$, then there exist $x\in\{a_1,a_2,\ldots, a_k,u\}$, $u_i$ and $f_j$ for some $i,j\in[k]$ such that the edges $u_if_j$ and $xu_i$ are incompatible (at $u_i$). Therefore as $(G, \mathcal{F})$ is $\mu n$-bounded, the number of bad mates $\mf{f}$ for $P_1$ (resp. for $P_2$) is at most $(k^2\mu n)n^{k-1}$. For $i\in[2]$, we define the family
\[\mathcal{M}_i=\{\mf{f}: \mf{f} \textrm{ is a good mate of } \mf{e}_i\}.\]
Then by the choice of $\mu\ll \be,\ga,\frac{1}{k}$, it is easy to see that $|\mathcal{M}_i|\ge \frac{\be}{2} n^k$ for $i\in[2]$ and by applying Lemma~\ref{connecting1} to $G$ with $(\ga/2,\be/2, W\cup V(P_1)\cup V(P_2)\cup \{u,v\})$ in place of $(\ga,\beta, W)$, we can obtain disjoint good mates $\mf{f}_i\in\mathcal{M}_i$ for $i\in[2]$ and a compatible $k$-th power of a path $Q$ of length at most $L$ whose ends are $\overf{f}_1$ and $\overf{f}_2$. It is easy to check that the $k$-th power of the paths \[a_1a_2\ldots a_k u_1u_2\ldots u_k\sim Q\sim v_k\ldots v_1b_k\ldots b_2b_1,\]\[a_1a_2\ldots a_k \sim u\sim u_1u_2\ldots u_k\sim Q\sim v_k\ldots v_1\sim v\sim b_k\ldots b_2b_1\] are all compatible in $(G, \mathcal{F})$.
\end{proof}
To this end, it is worth a remark that Lemma~\ref{connecting} enables us to complete the absorption of the left-over vertices. Moreover, instead of $\be$-absorbers, given any two vertex-disjoint compatible copies of $P_{2k}^k$ each of which has an end $\mf{e}_i$ with $M(\mf{e}_i)\ge \beta
n^{k},i\in[2]$, we can connect them into a compatible $k$-th power of a longer path using almost the same argument. The other situations regarding the presence of $u$ or $v$ follow as well and we omit this in the statement of Lemma~\ref{connecting}.
\subsection{Regularity}
An important ingredient in our proofs is Szemer\'edi's Regularity Lemma, and we first give the crucial notion of $\eps$-regular pairs. 
\begin{definition}\label{regularpair} (Regular pair). Given a graph $G$ and disjoint vertex subsets $X, Y
\subseteq V
(G)$, the \emph{density}
of the pair $(X, Y)$ is defined as $d(X, Y ):= \frac{e(X, Y )}{ |X||Y |}$,
where $e(X, Y ) := e(G[X, Y])$. For $\varepsilon > 0$, the pair $(X, Y )$ is \emph{$\varepsilon$-regular}
if for every $A \subseteq X, B\subseteq
Y$ with $|A| \geq\varepsilon|X|, |B|\ge \varepsilon|Y |$, we have
$|d(A, B) -d(X, Y )| < \varepsilon$.
Additionally, if $d(X, Y ) \ge d$ for some $d > 0$, then we say that $(X, Y )$ is \emph{$(\varepsilon,
d)$-regular}.
\end{definition}
\begin{definition}\label{regularpartition} (Regular partition). For a graph $G = (V, E)$ and $\varepsilon,
d > 0$, a
partition $V = V_0\cup
V_1\cup\ldots\cup V_k$ is \emph{$(\varepsilon, d)$-regular}, if

$\bullet$ $|V_0|\le \varepsilon|V|$;

$\bullet$ $|V_1| = |V_2| =\ldots = |V_k|\le \lceil\varepsilon|V |\rceil$;

$\bullet$ all but at most $\varepsilon k^2$ pairs $(V_i, V_j )$ with $1 \le i < j \le k$ are
$(\varepsilon,d)$-regular.
\end{definition}

Moreover, we usually call $V_1,\ldots, V_k$ \emph{clusters} and call $V_0$ the
\emph{exceptional set}.

\begin{fact}\label{fact1} Let $(X, Y )$ be an $(\varepsilon, d)$-regular pair, and $B\subseteq Y$ with
$|B|\geq
\varepsilon|Y|$. Then all but at most
$\varepsilon|X|$ vertices in $X$ have degree at least $(d-\varepsilon)|B|$ in $B$.
\end{fact}

\begin{fact}\label{slicing} (Slicing lemma, \cite{KSS}). Let $(X, Y )$ be an $(\varepsilon, d)$-regular
pair. Then for
any $\varepsilon\le \eta\le1$ and $X'\subseteq X, Y'\subseteq Y$ with $|X'|\ge \eta|X|, |Y'|\ge \eta|Y|$, the pair $(X',Y')$ is an
$(\varepsilon', d')$-regular pair with
$\varepsilon' = \max\{\varepsilon/\eta, 2\varepsilon\}$ and $d'=d-
\varepsilon$.
\end{fact}

\begin{lemma}\label{RL} (Degree form of the Regularity Lemma, \cite{KSS}). For every $\varepsilon > 0$,
there is an
$M= M(\varepsilon)$ such
that if $G = (V, E)$ is any graph and $d \in (0, 1]$ is any real number, then there is an $(\varepsilon,
d)$-regular
partition $V = V_0\cup
V_1\cup \ldots \cup V_k$ with $|V_i| = m$ for each $i\in[k]$, and a spanning subgraph $G'\subseteq G$ with
the
following properties:

$\bullet$ $1/\varepsilon \le k \le M$;

$\bullet$ $d_{G'}(v) > d_G(v)-(d + \varepsilon)|V |$ for all $v \in V$;

$\bullet$ $e(G'[V_i]) = 0$ for all $i\ge 1$;

$\bullet$ all pairs $(V_i, V_j) (1 \le i < j \le k)$ are $\varepsilon$-regular in $G'$ with density 0 or
at least $d$.
\end{lemma}

\begin{definition}\label{reducedgraph} (Reduced graph). Given an arbitrary graph $G = (V, E)$, a partition
$V =
V_1\cup\ldots\cup V_k$, and two
parameters $\varepsilon, d > 0$, the \emph{reduced graph} $R = R(\varepsilon, d)$ of $G$ is defined as
follows:

$\bullet$ $V (R) = [k]$,

$\bullet$ $ij \in E(R)$ if and only if $(V_i, V_j )$ is $(\varepsilon, d)$-regular.

\end{definition}

As remarked in \cite{KSS}, a typical application of degree form of the Regularity Lemma begins with a
graph $G = (V,
E)$ and appropriate
parameters $\varepsilon, d > 0$, and then obtains an $(\varepsilon, d)$-regular
partition $V = V_0\cup V_1\cup \ldots \cup V_k$  and a subgraph $G'$ with above-mentioned properties. Then
we usually drop the exceptional set $V_0$ to get a \emph{pure graph} $G'' = G'-V_0$ of $G$ and study the
properties of reduced graph $R = R(\varepsilon, d)$ of $G''$. By Lemma \ref{RL},

$$\delta(R)\ge\frac{\delta(G)-(d+\varepsilon)|V|-|V_0|}{m}\ge\frac{\delta(G)-(d+2\varepsilon)|V|}{m}.$$
In particular, if $\delta(G) \geq c|V |$, then $\delta(R)\ge(c-d-2\varepsilon)|R|$.

We also need a `compatible' variant of graph counting lemma as follows.
\begin{lemma}\label{hlwy} \cite{hu2022graph}
For constant $d,\eta>0$ and positive integers $r,h_1,\ldots,h_r$ with $\sum_{i=1}^rh_i=:h$, there exist positive constants $\varepsilon^\ast=\varepsilon^\ast(r,d,h), c=c(r,d,h)$ and $\mu=\mu(r,d,h,\eta)$ such that the following holds for sufficiently large $n$. Let $(G,\mathcal{F})$ be a $\mu n$-bounded incompatibility system with $|G|=n$ and $U_1,\ldots,U_r$ be pairwise vertex-disjoint sets in $V(G)$ with $|U_i|\ge \eta n$, $i \in [r]$ and every pair $(U_i,U_j)$ being $(\varepsilon^\ast,d)$-regular. Then there exist at least $c\prod_{i=1}^r|U_i|^{h_i}$ compatible copies of $K_r(h_1,\ldots,h_r)$ in $G[U_1,\ldots,U_r]$, each containing exactly $h_i$ vertices in $U_i$ for every $i\in[r]$.
\end{lemma}

\begin{corollary}\label{co1} For any $d, \eta > 0$ and integer $h \ge 1$, and a graph $H$ with $V(H)=\{u_1,\ldots,u_h\}$, there exist positive constants $\varepsilon^*=\varepsilon^*(h,d), c = c(h,d)$ and $\mu=\mu(h, d,\eta)$ such that the following holds for sufficiently large $n$. Let $(G, \mathcal{F})$ be a $\mu n$-bounded incompatibility
system with $|G|=n$, and  $U_1,\ldots,U_h$ be pairwise vertex-disjoint sets in $V(G)$ with $|U_i| \ge \eta n$, $i \in [h]$, and $(U_i,U_j)$ are $(\varepsilon^*, d)$-regular if $u_iu_j\in H$, where $\{i,j\}\subseteq [h]$.
Then  $G$ contains $c\prod^h_{i=1}|U_i|$ compatible copies of $H$, and all the corresponding vertices of $u_i$ are in $U_i$, where $i\in[h]$.
\end{corollary}

To this end, we give two well-known concentration inequalities for random variables.
\begin{lemma}\label{chernoff}(Chernoff's inequality, \cite{janson}, Corollary 2.3). Let $X\sim Bin(n,p)$, Then for every $0<a<3/2$, we have
$$\mathbb{P}(|X-\mathbb{E}X|>a\mathbb{E}X)< 2e^{-a^2\mathbb{E}X/3}.$$
\end{lemma}

\begin{lemma}\label{Janson}(Janson's inequality, \cite{janson}, Theorem 2.14). Let $p\in[0,1]$, $G$ be a graph and $R$ be a random vertex subset obtained by including every vertex of $G$ independently with probability $p$. Let $\mathcal{F}\subseteq 2^{V(G)}$ be a collection of vertex subsets of $G$. Given a vertex subset $F$ in $\mathcal{F}$, we denote by $I_F$ the indicator random variable which is 1 if $F$ is contained in $R$ and 0 otherwise.
Let $X=\sum_{F\in\mathcal{F}}I_F$, let $\lambda=\mathbb{E}[X]$ and let $$\overline{\Delta}=\sum_{(F,F')\in \mathcal{F}^2:F\cap F'\neq\emptyset}\mathbb{E}[I_FI_{F'}].$$ Then, for every $\varepsilon\in(0,1)$, we have $$\mathbb{P}[X\leq(1-\varepsilon)\lambda]\leq \exp\big(-\frac{\varepsilon^2\lambda^2}{2\overline{\Delta}}\big).$$

\end{lemma}

\subsection{Put things together}\label{Pf of main}

\begin{proof}[Proof of Theorem~\ref{main theorem}]
For any $\gamma > 0$, we choose $$\frac{1}{n}\ll\mu\ll\lambda,\beta_3\ll\tau\ll p\ll\frac{1}{L},\beta_1,\beta_2\ll \gamma,\frac{1}{k},$$ and let $(G, \mathcal{F})$ be an
$(n, \frac{k}{k+1} + \gamma, \mu)$-incompatibility system.
By Lemma \ref{absorbing},  there exist
$\beta_1,
\beta_2> 0$ such that every vertex of $G$ has $\beta_1n$
vertex-disjoint
$\beta_2$-absorbers. 
Let $R$ be a random set of vertices obtained by including every vertex of $G$ independently with
probability $p$.
\begin{claim}[Reservoir]\label{R} The following properties hold with high probability.
  \stepcounter{propcounter}
\begin{enumerate}[label = ({\bfseries \Alph{propcounter}\arabic{enumi}})]
  \item\label{e1} $\frac{1}{2}pn\le |R|\le\frac{3}{2}pn$;
  \item\label{e2} $d_R(v)\ge(\frac{k}{k+1}+\frac{\gamma}{2})|R|$ for each $v\in V(G)$;
  \item\label{e3} Every vertex $v\in V(G)$ has at least $\frac{p^{2k}}{2}|\mathcal{A}(v)|$ vertex-disjoint $\beta_2$-absorbers in $R$, where $\mathcal{A}(v)$ is a maximum set of vertex-disjoint $\beta_2$-absorbers of $v$ in $(G,\mathcal{F})$;
  \item\label{e4} Given any fixed constant $\be>0$, every $k$-tuple $\mf{e}$ with $M(\mf{e})\ge\beta n^k$ has at least $\frac{p^k}{2}M(\mf{e})$ mates in $R$, where $M(\mf{e})$ is the number of the mates of $\mf{e}$ in $(G,\mathcal{F})$.
\end{enumerate}

\end{claim}

\begin{proof}

  (1) Note that $\mathbb{E}[|R|]=pn$. By Lemma \ref{chernoff}, we have $\mathbb{P}(||R|-pn|>\frac{pn}{2})<\exp(-\frac{1}{12}\mathbb{E}[|R|])=\exp(-pn/12)$.


(2) Choose $\delta\ll\gamma$ (need $\delta<\frac{\gamma(k+1)}{4k+3\gamma k+3\gamma}$). We have $\mathbb{P}(|R|>(1+\delta)pn)<\exp(-\delta^2 pn/3)$ by Lemma \ref{chernoff}. On the other hand, for a given vertex $v\in V(G)$, $\mathbb{E}[d_R(v)]\ge(\frac{k}{k+1}+\gamma)pn$. So
$\mathbb{P}\big(d_R(v)<(\frac{k}{k+1}+\gamma)pn(1-\delta)\big)<\exp\big(-\delta^2(\frac{k}{k+1}+\gamma)pn/3\big)$ by Lemma \ref{chernoff}. Hence
\begin{eqnarray*}
  \mathbb{P}\big(d_R(v)<(\frac{k}{k+1}+\frac{\gamma}{2})|R|\big) &<&  \mathbb{P}\big(d_R(v)<(\frac{k}{k+1}+\gamma)\frac{1-\delta}{1+\delta}|R|\big)\\ &\le&\mathbb{P}\big(d_R(v)<(\frac{k}{k+1}+\gamma)pn(1-\delta)\big) +\mathbb{P}(|R|>(1+\delta)pn)\\
   &<& \exp\big(-\delta^2(\frac{k}{k+1}+\gamma)pn/3\big)+\exp(-\delta^2 pn/3).
\end{eqnarray*}

(3) For every vertex $v\in V(G)$, let $X_v$ be the number of vertex-disjoint $\beta_2$-robust absorbers from $\mathcal{A}(v)$ that lie in $R$. Then $\mathbb{E}[X_v]= p^{2k}|\mathcal{A}(v)|$. Note that $|\mathcal{A}(v)|\ge\beta_1 n$.
Hence $\mathbb{P}(X_v<\frac{p^{2k}}{2}|\mathcal{A}(v)|)=\mathbb{P}(X_v<\mathbb{E}[X_v]/2)<\exp(-\frac{\mathbb{E}[X_v]}{12})
<\exp(-\frac{p^{2k}}{12}|\mathcal{A}(v)|)<\exp(-p^{2k}\beta_1n/12)$ by Lemma \ref{chernoff}.

(4) For a given $k$-tuple $\mf{e}$ with $M(\mf{e})\ge \beta n^k$, let $\mathcal{F}$ be the set of mates for $\mf{e}$. Then $|\mathcal{F}|=M(\mf{e})\ge\be n^k$. We choose a subfamily $\mathcal{F}'$ of $\mathcal{F}$ with exactly $\be n^k$ members. For each $F\in \mathcal{F}'$, let $I_F$ be the indicator random variable with
\begin{displaymath}
I_F = \left\{ \begin{array}{ll}
1 & \textrm{if $F$ is contained in $R$,}\\
0 & \textrm{otherwise,}
\end{array} \right.
\end{displaymath}
and let $X=\sum_{F\in\mathcal{F}'}I_F$. We have $\mathbb{E}[X]= \be n^kp^k$, and
\begin{eqnarray*}
  \overline{\Delta}-\mathbb{E}[X] &\le& \sum^{k-1}_{s=1}{k\choose s}|\mathcal{F}'|\cdot n^{k-s}\cdot p^{2k-s} \\
   &=& \sum^{k-1}_{s=1}{k\choose s}\be (np)^{2k-s} \\
   &\le& k^k\be (np)^{2k-1}.
\end{eqnarray*}
By Lemma \ref{Janson} we have that $\mathbb{P}[X<\frac{p^k}{2}M(\mf{e})]=\mathbb{P}[X<\mathbb{E}[X]/2]\le\exp\Big(-\frac{(\mathbb{E}[X])^2}{8\overline{\Delta}}\Big)\le\exp\Big(-\frac{\be np}{8k^k}\Big)$.

By the union bound, we obtain that with high probability, \ref{e1}--\ref{e4} hold.
\end{proof}

\noindent\textbf{Connecting $\be_2$-absorbers.}
We choose the reservoir $R$ as above. Note that by Lemma \ref{absorbing} each vertex has at least $\be_1n$ vertex-disjoint $\be_2$-absorbers. Then by the choice of $p\ll\beta_1,\beta_2$ and Claim~\ref{R}~\ref{e1}, we can greedily pick a collection $\mathcal{C}=\{A_v:v\in R\}$ of vertex-disjoint $\be_2$-absorbers $A_v$ (for $v$) from $V(G)\setminus R$. Moreover, $\de(G-R)\ge (\frac{k}{k+1} + \frac{\gamma}{2})n$ and every end of $A_v$ has at least $\beta_2n^k-|R|n^{k-1}\ge \frac{\beta_2}{2}n^k$ mates in $G-R$. Let $\mathcal{S}$ be an arbitrary sequence of all these copies $A_v,v\in R$. By repeatedly applying Lemma \ref{connecting}, we can iteratively connect two consecutive copies from $\mathcal{S}$ into a compatible $k$-th power of a path, say $P_R$, via a collection of $|R|-1$ vertex-disjoint compatible $k$-th power of paths of length at most $L$ in $G-R$. Indeed, during the process, suppose we have two consecutive absorbers $A_u,A_v$ to be connected and let $\mathcal{Q}$ be the family of vertex-disjoint $k$-th power of paths used for previous connections along the sequence. Since $p\ll \frac{1}{L},\be_2\ll\ga$ and thus $|R|+|V(\mathcal{Q})|\le (L+1)\cdot\frac{3}{2}pn\le \min\{\frac{\gamma}{8}n,\frac{\be_2}{8}n\}$, we can apply Lemma~\ref{connecting} to $G-R$ with $(\frac{\ga}{2},\frac{\be_2}{2}, R\cup V(\mathcal{Q}))$ in place of $(\ga,\be, W)$ and obtain a compatible $k$-th power of path $Q_{uv}$ satisfying \ref{abs}. 

\noindent\textbf{Almost cover.}
 Let $G'=G[V(G)-R-V(P_R)]$.
Then $\de(G')\ge \de(G)-(L+2k)\frac{3}{2}pn\ge(\frac{k}{k+1}+\frac{\gamma}{2})n$. By Lemma \ref{covering}, $G'$ contains a family of vertex-disjoint
compatible $k$-th power of paths $P_1,\ldots, P_s$ for some $s\le \frac{1}{\la}$, each of length at least $\lambda n$, that cover all but at most $\tau n$
vertices of $G'$, and all the ends of them have at least $\beta_3 n^k$ mates in $V(G)$ and thus at least $\frac{\beta_3p^{k}}{2} n^k$ mates in $R$ by Claim~\ref{R}~\ref{e4}. By repeatedly applying Lemma~\ref{connecting} to $R$ with $(\frac{\ga}{2},\frac{\beta_3p^{k}}{2})$ in place of $(\ga,\beta)$ as above, we can iteratively connect $P_1,\ldots, P_s$ into a compatible $k$-th power of a path, say $P_{G'}$, via vertex-disjoint $k$-th power of paths of length at most $L$ in $R$. This can be done because in the process, we need to avoid a set of at most $(s-1)L<\min\{\frac{\gamma}{8}|R|,\frac{\be_3p^k}{8}|R|\}$ vertices that have been used in previous connections.

\noindent\textbf{Connecting the sandwich.}  Let $T$ the set of at most $\tau n$ vertices not covered by the paths $P_1,\ldots,P_s$ and write $T=\{v_1,v_2,\ldots,v_t\}$, where $t=|T|\le\tau n$. Note that $S:=R\cap V(P_{G'})$ has size $|S|\le sL\le \frac{L}{\la}$.
Thus by Lemma \ref{absorbing} and Claim \ref{R}~\ref{e3}, every vertex in $T$ has at least $\beta_1p^{2k}n/2-|S|\ge \beta_1p^{2k}n/4$ vertex-disjoint $\beta_2$-absorbers in $R\setminus V(P_{G'})$. We assign $v_1$ in $T$ a $\be_2$-absorber in $R\setminus V(P_{G'})$, say $A_1$ and denote by $A_1'$ the corresponding compatible copy of $P^{k}_{2k+1}$ induced by $A_1\cup\{v_1\}$, each end of which has at least $\beta_2n^k$ mates in $G$. Also, by applying Lemma~\ref{connecting} to $R$ twice, we can connect $P_R,P_{G'},A'_1$ in order and obtain a compatible $k$-th power of a path, say $P_{comb}$. This is because we only need to avoid a constant number of vertices used in previous connections. Note that both ends of $P_{comb}$ have at least $\beta_2n^k$ mates in $G$.

\noindent\textbf{Connecting the rest.} Note that $S_1:=R\cap V(P_{comb})$ has $|S_1|\le (s+2)L$. Then every vertex in $T$ has at least $\beta_1p^{2k}n/2-|S_1|\ge \beta_1p^{2k}n/4$ vertex-disjoint $\beta_2$-absorbers in $R\setminus V(P_{G'})$.
Now, since $|T|\le \tau n$ and $\tau\ll p,\be_1$, we can greedily assign every vertex $v_i$ in $T\setminus\{v_1\}$ a $\be_2$-absorber in $R\setminus S_1$, say $A_i$, such that these absorbers are vertex-disjoint. Similarly for every $i\in[2,t]$, we denote by $A_i'$ the corresponding compatible copy of $P^{k}_{2k+1}$ induced by $A_i\cup \{v_i\}$, each end of which has at least $\frac{\beta_2p^{k}}{2}n^k$ mates in $R$.

Next we shall connect $P_{comb},A'_2,\ldots,A_t'$ into compatible $k$-th power of a cycle.  Recall that $t\le \tau n$, $d_{R}(v)\ge(\frac{k}{k+1}+\frac{\gamma}{2})|R|$ for each $v\in V(G)$ and
 every end of $P_{comb},A'_2,\ldots,A_t'$ has at least $\frac{\beta_2p^k}{2}n^k$ mates in $R$.
Again, by applying Lemma \ref{connecting} to $R$ 
we iteratively connect $P_{comb},A'_2,\ldots,A_t'$ into compatible $k$-th power of a cycle via a collection of $t$ vertex-disjoint compatible $k$-th power of paths of length at most $L$. In fact, in each step $i, i\in[t]$, since $\tau\ll p,\be_2,\ga$ and thus the number of vertices in $R$ covered in previous steps is at most $|S_1|+iL\le 2L\tau n<\min\{\frac{\ga}{8}|R|, \frac{\be_2p^k}{8}|R|\}$, by Lemma \ref{connecting} one can obtain a path as desired. Let $C$ be the resulting cycle.
Finally, using property \ref{abs},
we absorb the leftover vertices in $V(G)-V(C)$ by the
absorbers $A_v$ in $\mathcal{C}$ and obtain
a compatible $k$-th power of Hamilton cycle, as desired.
\end{proof}

\section{Absorbing lemma} \label{sectionabsorbing}

\emph{Proof of Lemma \ref{absorbing}.} Given $\gamma>0$, we choose $$\frac{1}{n}\ll\mu, \be_1,\be_2\ll\eps,c\ll
d \ll\gamma, \frac{1}{k}.$$
For every $v$, we use $G_v$ to denote $G[N(v)]$. Firstly we apply Lemma \ref{RL} on $G_v$ with density $d$ and
obtain an
$(\varepsilon,d)$-regular partition $V (G_v) =V_0 \cup V_1 \cup\ldots\cup V_r$ for some $1/\varepsilon \le r \le M$
and $|V_i| =
m
\ge\frac{(1-\varepsilon)|G_v|}{r}$ for each $i \in [r]$. Let $R = R(\varepsilon, d)$
be the reduced graph for this partition. Note that

\begin{eqnarray*}
\delta(G_v) &\ge&(\frac{k}{k+1}+\gamma)n-(n-|G_v|) \\
 &\ge& |G_v|-\frac{1}{k+1}n+\gamma n \\
 &\ge& (1-\frac{1}{k})|G_v|+\gamma n
\end{eqnarray*}
The last inequality follows since $|G_v|\ge (\frac{k}{k+1}+\gamma)n$.  As $\varepsilon\ll d\ll\gamma$, we have
$\delta(R) \geq\big(1-\frac{1}{k}+\gamma-d-2\varepsilon\big)r\ge\big(\frac{k-1}{k} +
\frac{\gamma}{2}\big)r$.
It follows that every $k$ vertices in $R$ have at least $\gamma r$ common neighbors. Hence we can greedily find a copy of $P^k_{2k}$ in $R$, denoted by $H$. Let $V(H)=\{u_1,\ldots, u_{2k}\}$, and $V_1,\ldots, V_{2k}$
be the
corresponding clusters.  By Definition \ref{reducedgraph}, $(V_i, V_j)$ is $(\varepsilon,d)$-regular if $u_iu_j$ is an edge of $H$ in $R$. By Corollary \ref{co1}, there exists $c>0$ such that $G_v[V_1\cup\ldots\cup V_{2k}]$ contains at least $cm^{2k}\ge c'n^{2k}$
compatible copies of $P^k_{2k}$, where $c'$ is a constant depending on $\eps$ and $c$.

We claim that for any $0<\beta_2<c'/2$, there exists a family $\mathcal{P}$ of at least $\frac{c'}{2}n^{2k}$ compatible copies of $P^k_{2k}$ such that every end of
them has at least $\beta_2 n^{k}$
mates in $G$. For otherwise, let $x<\frac{c'}{2}n^{2k}$ be the number of compatible copies of $P^k_{2k}$ such that every end of
them has at least $\beta_2 n^{k}$ mates. Then the number of $k$-tuples which have at least $\beta_2 n^{k}$ mates is at least $\frac{x}{n^k}$, and so the number of $k$-tuples which have less than $\beta_2 n^{k}$ mates is at most $n^k-\frac{x}{n^k}$. Hence the total number of compatible copies of $P^k_{2k}$ in $G_v[V_1\cup\ldots\cup V_{2k}]$ is at most

$$x+(n^{k}-\frac{x}{n^k})\cdot\beta_2n^k<(\frac{c'}{2}+\beta_2-\frac{c'}{2}\beta_2)n^{2k}<c'n^{2k},$$
a contradiction.

Among those copies in $\mathcal{P}$, at least $(\frac{c'}{2}-2\mu)n^{2k}$ of them are $\beta_2$-absorbers for $v$. Indeed, since $(G, \mathcal{F})$ is an $(n,\frac{k}{k+1} + \gamma, \mu)$-incompatibility system,
there are at
most $\mu n^2$
incompatible pairs ${vv_1, vv_2}$ and for each such pair ${vv_1, vv_2}$, there are at most $n^{2k-2}$
copies of $P^k_{2k}$ containing $v_1$ and
$v_2$.  Moreover, at most $\mu n^2\cdot n^{2k-2}$ copies contain an edge $v_1v_2$ which is incompatible with
$vv_1$. Hence we obtain at least
$(\frac{c'}{2}-2\mu)n^{2k}$ $\beta_2$-absorbers for $v$.
As $\mu,\beta_1\ll \eps,c$, we can greedily find $\beta_1 n$ vertex-disjoint
$\beta_2$-absorbers for $v$.

\section{Almost cover}\label{sectioncovering}

\pr[Proof of Lemma \ref{covering}] Given $\gamma,\tau>0$ and $k\in\N$, we choose \[\frac{1}{n}\ll\mu,\be,\lambda\ll\eps,c\ll
\eta,d\ll\gamma,\tau,\frac{1}{k}\] and fix $(G, \mathcal{F})$ to be an $(n,\frac{k}{k+1} + \gamma, \mu)$-incompatibility system.
We apply Lemma \ref{RL} on $G$ to obtain an
$(\varepsilon,d)$-regular partition $V (G) =V_0 \cup V_1 \cup\ldots\cup V_r$ for some $1/\varepsilon \le r
\le M$ and $|V_i|
=m\ge\frac{(1-\varepsilon)n}{r}$ for each $i \in [r]$. Let $R = R(\varepsilon, d)$ be the reduced graph
for this partition. Since $\varepsilon\ll
d\ll\gamma$, we have
$\delta(R) \geq\big(\frac{k}{k+1} + \gamma-d-2\varepsilon\big)r\ge\big(\frac{k}{k+1} +
\frac{\gamma}{2}\big)r$. By the Hajnal--Szemer\'{e}di theorem \cite{HS}, $R$ has a $K_{k+1}$-tiling
$\mathcal{K} = \{K^{(1)},K^{(2)},\ldots,
K^{\lfloor\frac{r}{k+1}\rfloor}\}$ covering all but at most $k$ vertices. Next, we shall pick vertex-disjoint compatible $k$-th power of long paths within the corresponding $k+1$ clusters for every $K^{(i)}\in \mathcal{K}$. Without loss of generality, we may take $K^{(1)}$ for instance and assume that the corresponding clusters of $K^{(1)}$ are $V_1,\ldots, V_{k+1}$. It suffices to prove the following result.
\begin{claim}\label{c1}
  For any collection of subsets $U_i\<V_i, i\in[k+1]$ with $|U_i|\ge \frac{\eta}{2}m$, there exists, in $G[U_1,\ldots, U_{k+1}]$, a compatible copies of $P^k_{s}$ which satisfies the following properties:
  \stepcounter{propcounter}
\begin{enumerate}[label = ({\bfseries \Alph{propcounter}\arabic{enumi}})]
       \item\label{p1} $s\in (k+1)\N$ and $s\ge \lambda n$;
       \item\label{p2} both ends have at least $\be n^k$ mates.
\end{enumerate}
\end{claim}
Thus taking this for granted, we can greedily pick a family of vertex-disjoint compatible copies of $P^k_{s}$, which altogether leave less than $\frac{\eta}{2}m$ vertices in each $V_i$ ($ i\in [k+1]$) uncovered. Applying this for every $K^{(i)}\in \mathcal{K}$, we obtain a desired family of vertex-disjoint compatible copies of $P^k_{s}$ and the number of vertices uncovered is at most \[km+\frac{\eta}{2}m\cdot r+\varepsilon n<\tau n.\]

\pr[Proof of Claim~\ref{c1}]\renewcommand*{\qedsymbol}{$\blacksquare$}
By Corollary~\ref{co1} applied with $h=k+1$, there exists $c=c(k,d)>0$ such that the $(k+1)$-partite graph $G[U_1,\ldots,U_{k+1}]$ contains at least $c(\frac{\eta}{2}m)^{k+1}\ge c_1 n^{k+1}$ compatible copies of $K_{k+1}$ for a constant $c_1>0$.
We construct an auxiliary $(k+1)$-uniform $(k+1)$-partite hypergraph $H$ on $V(H):=U_1\cup\ldots\cup U_{k+1}$ with
\[E(H):=\big\{\{v_1,v_2,\ldots,v_{k+1}\}: v_i\in V_i,~i\in[k+1],\textrm{ and } v_1,\ldots, v_{k+1}\textrm{ induce a compatible copy of }
K_{k+1}\big\}.\]
Thus $H$ has at least $c_1n^{k+1}$ hyperedges.
Now we claim that $H$ has a subgraph $H'$ such that every $k$-set is contained in either at least $\frac{c_1}{2}n$ hyperedges or zero hyperedge. In fact, we can proceed by iteratively removing the edges from $H$ as follows. If a $k$-set $S$ is contained in less than $\frac{c_1}{2} n$ hyperedges in the current hypergraph, then we remove all hyperedges containing $S$. Note that in the process, we remove at most $\frac{c_1}{2}n\cdot m^k\cdot(k+1)\le\frac{c_1}{2}n^{k+1}$ edges in total. The process terminates at a nonempty hypergraph $H'$ as desired.

Now we pick a longest compatible $k$-th power of a path, denoted as $P=u_1\ldots u_{s'}$ in $G[U_1,\ldots, U_{k+1}]$, such that for every $i\in[s'-k]$, $\{u_i,\ldots,u_{i+k}\}$ is a hyperedge of $H'$. Suppose for contradiction that $s'<2\lambda n$. Note that in $H'$, the $k$-set $S=\{u_{s'-k+1},\ldots, u_{s'}\}$ has a set of at least $\frac{c_1}{2} n-s'\ge \frac{c_1}{4} n$ neighbors in $V(H')\setminus V(P)$, denoted as $A$. In particular, as $\mathcal{F}$ is $\mu n$-bounded, there are at most $k^2\mu n$ vertices $v$ such that for some $i\in\{s'-k+1,\ldots,s'\}$ and $j\in\{i-k,\ldots,i-1\}$, the edge $u_{i}v$, $u_{i}u_{j}$ are compatible at $u_i$. By the choice of $\mu\ll \eps$ and thus $|A|>k^2\mu n$, we can pick a vertex from $A$ to extend $P$ as required, a contradiction. Thus by consecutively removing vertices from one end of $P$, we can obtain a compatible $k$-th power of a subpath on $s$ vertices with $s\in (k+1)\N$ and $s\ge \lambda n$.

It remains to prove \ref{p2}. In fact, we can extend every $k$-set from a hyperedge to at least $(\frac{c_1}{4}n-k^2\mu n)^k\ge \be n^k$ compatible copies of $P^k_{2k}$ as above. This is because in each step, we have at least $\frac{c_1}{4}n-k^2\mu n$ choices for the next vertex. This completes the entire proof.
\oof
\oof
\section{Connecting two ends}\label{sectionconnecting}
We will make use of a result of Koml\'{o}s--S\'{a}rk\"{o}zy--Szemer\'{e}di~\cite{Komlos1998On}.
\begin{lemma}\label{KSS98}\cite{Komlos1998On} For every $\ga>0$ and $k\in \N$, there exists $L\in \N$ such that the following holds for sufficiently large $n\in \N$. Let $R$ be an $n$-vertex graph with $\delta(R)\ge(\frac{k}{k+1}+\ga)n$ and $\mf{e}_1,\mf{e}_2$ be two disjoint $k$-tuples of vertices, each of which induces a copy of $K_k$. Then there exists a $k$-th power of a path $P$ of length
at most $L$, whose ends are $\overf{e}_1$ and $\overf{e}_2$.
\end{lemma}
\pr[Proof of Lemma \ref{connecting1}] For any $\beta, \gamma > 0$, we choose \[\frac{1}{n}\ll\mu\ll\eps,c\ll d\ll\beta,\gamma,\frac{1}{k},~\text{and additionally}~\frac{1}{L}\ll \ga .\]
We fix
$(G, \mathcal{F})$ to be an $(n,\frac{k}{k+1} + \gamma, \mu)$-incompatibility system, and $W\subseteq
V(G)$
with $|W|<\min\{\frac{\beta}{2}n,\frac{\gamma}{2}n\}$, and $\mf{e}_i,\mathcal{M}_i$ with $|\mathcal{M}_i|\ge \beta
n^{k}, i=1,2$. Without loss of generality, we write
\[ ~\mf{e}_1=(u_1,\ldots,u_k)~\text{and}~\mf{e}_2=(v_1,\ldots,v_k).\]
Let $V'=V(G)\setminus (W\cup V(\mf{e}_1)\cup V(\mf{e}_2))$. For $i\in[2]$, we define the family
\[\mathcal{H}_i=\{\mf{f}\in \mathcal{M}_i: \mf{f}~\text{lies inside}~V'\}.\]
Then by the choice of $|W|<\min\{\frac{\gamma}{2}n,\frac{\beta}{2}n\}$, it is easy to see that $|\mathcal{H}_i|\ge \frac{\be}{4} n^k$ for $i\in[2]$. We then uniformly and randomly partition $V'$ into $2k$ parts of nearly equal size, denoted as $U_{1},U_{2},\ldots, U_{2k}$. Let $\mathcal{X}_1$ (or $\mathcal{X}_2$) be the family of mates $\mf{f}=(f_1,\ldots,f_k)$ in $\mathcal{H}_1$ (resp. in $\mathcal{H}_2$) with $f_i\in U_{i}$ for $i\in[k]$ (resp. $f_i\in U_{k+i}$ for $i\in[k]$). It follows that $\mathbb{E}(|\mathcal{X}_i|)\ge\frac{\beta}{4(2k)^k} n^k$. Then by a standard application of Janson's inequality\footnote{Similar to the proof of Claim~\ref{R}}, there exists a partition such that $|\mathcal{X}_i|\ge \frac{1}{2}\mathbb{E}(|\mathcal{X}_i|)\ge\frac{\beta}{8(2k)^k} n^k$ for every $i\in[2]$. Let $U_i'=V(\mathcal{X}_1)\cap U_i$ and $U_{k+i}'=V(\mathcal{X}_2)\cap U_{k+i}$ for every $i\in[k]$. Then it is easy to observe the following properties:
  \stepcounter{propcounter}
\begin{enumerate}[label = ({\bfseries \Alph{propcounter}\arabic{enumi}})]
       \item\label{p3} $|U'_i|\ge\frac{\beta}{8} n$ for every $i\in[2k]$;
       \item\label{p4} Combined with $\mf{e}_1$, every copy of $K_k$ in the resulting $k$-partite graph $G[U_1',U_2',\ldots,U_k']$ can form a copy of $P^k_{2k}$. The same assertion also holds for $\mf{e}_2$ and $G[U_{k+1}',U_{k+2}',\ldots,U_{2k}']$.
       \item\label{p41} Moreover, for every $i\in[k]$ and $x\in U_i'$, the edges $xu_i,xu_{i+1},\ldots,xu_k$ are pairwise compatible at $x$; similarly, for every $j\in[k]$ and $y\in U_{k+j}'$, the edges $yv_j,yv_{j+1},\ldots,yv_k$ are pairwise compatible at $y$.
\end{enumerate}
It is worth a remark that \ref{p41} follows from the existence of a mate $\mf{f}$ for $\mf{e}_1$ (resp. for $\mf{e}_2$) whose $i$-th coordinate is $x$.

Let $G'=G[V']$, and $n'=|V'|$. Then $\de(G')\ge \de(G)-|W|-2k\ge (\frac{k}{k+1} + \frac{\gamma}{3})n$.
We apply Lemma \ref{RL} to $G'$ to refine the current partition $\{U'_{1},U'_{2},\ldots, U'_{2k}, V'\setminus \bigcup_{i=1}^{2k}U'_i\}$. Denote the resulting $(\eps,d)$-regular partition by $\mathcal{P} =\{V_0, V_1,\ldots,V_r\}$ for some $1/\varepsilon \le r
\le M_{\eps}$, where $|V_0|\le \eps n$ and all clusters $V_i$ with $i\in[r]$ have the same size, denoted as $m$.
Let $R = R(\varepsilon, d)$ be the reduced graph for this partition. Then as $\varepsilon\ll d\ll\gamma$, we obtain that
$\delta(R) \geq\big(\frac{k}{k+1}+\frac{\gamma}{3}-d-2\varepsilon\big)r\ge\big(\frac{k}{k+1} +
\frac{\gamma}{4}\big)r$.

We claim that there exists a copy of $K_k$ in $R$, such that each of the corresponding $k$ clusters comes from a different part $U_i'$ where $i\in[k]$. For otherwise, by removing the edges (of $G'$) between irregular pairs, pair with density less than $d$, and pairs incident with $V_0$, we obtain a subgraph of the $k$-partite graph $G[U_1',U_2',\ldots,U_k']$ which does not contain any copy of $K_k$. Thus the total number of copies of $K_k$ in $G[U_1',U_2',\ldots,U_k']$ is at most
\[(\varepsilon r^2\cdot m^2+r^2\cdot d\cdot m^2+\varepsilon n'\cdot n) n'^{k-2}\le(2\varepsilon +d)n^k<\frac{\beta}{8(2k)^k} n^k\le |\mathcal{X}_1|,\]
yielding a contradiction. Similarly, we can find another copy of $K_k$ in $R$, such that each of the corresponding $k$ clusters comes from a different part $U_i$ where $i\in[k+1,2k]$. Without loss of generality, we may then assume that $V_i\<U_i'$ for every $i\in[2k]$ and $\mf{f}_1:=(V_1,V_2,\ldots, V_k)$, $\mf{f}_2:=(V_{k+1},V_{k+2},\ldots, V_{2k})$ induce two copies of $K_k$ in $R$.

By Lemma~\ref{KSS98} applied to $R$,  there exists a copy of $P^k_{\ell}$ with $2k\le\ell\le L$, whose ends are $\overf{f}_1$ and $\overf{f}_2$. Suppose $V_1,\ldots,V_{\ell}$ are the corresponding clusters in its base path.  Applying Corollary \ref{co1} to $G'$ with $h=\ell, H=P^k_{\ell}$, $U_i=V_i$ for every $i\in[\ell]$ and combining the choice of $\mu\ll\eps\ll d$, we obtain a family $\mathcal{K}$ of at least $cm^{\ell}$ 
compatible copies of $P^k_{\ell}$ (in $G'$) for some $c=c(\ell,d)>0$. Moreover, every two ends of such a copy in $\mathcal{K}$ respect the orderings of $\overf{f}_1$ and $\overf{f}_2$, respectively. By property~\ref{p4}, we shall pick a copy of $P^k_{\ell}$ in $\mathcal{K}$ such that we can extend it to a desired compatible copy of $P_{\ell+2k}^k$ whose ends are $\overf{e}_1$ and $\overf{e}_2$, which completes the entire proof.

Recall that $\mf{e}_1=(u_1,\ldots,u_k)$ and $\mf{e}_2=(v_1,\ldots,v_k)$. For any fixed copy of $P^k_{\ell}$ in $\mathcal{K}$, its base path is denoted as $Q=w_1w_2\ldots w_{\ell}$ such that the path $u_1\ldots u_kw_1w_2\ldots w_{\ell}v_k\ldots v_1$ forms a copy of $P_{\ell+2k}^k$. Note that $w_i\in U_i'$ and $w_{\ell-i+1}\in U_{k+i}'$ for every $i\in[k]$. Combining \ref{p41} and the fact of $\ell\ge 2k$, we obtain that if $Q$ fails to connect the ends $\overf{e}_1$ and $\overf{e}_2$ into a compatible copy of $P_{\ell+2k}^k$, then one of the following conditions holds:
  \stepcounter{propcounter}
\begin{enumerate}[label = ({\bfseries \Alph{propcounter}\arabic{enumi}})]
       \item\label{p5} There exist some $u_i$ ($i\in[k]$) and a pair $\{w_s,w_t\}$ with $s,t\in[\ell]$ such that either $\{u_iw_s, u_iw_t\}$, or $\{u_iw_s, w_sw_t\}$ forms an incompatible pair in $(G,\mathcal{F})$.
       \item\label{p6} There exist some $v_j$ ($j\in[k]$) and a pair $\{w_s,w_t\}$ with $s,t\in[\ell]$ such that either $\{v_jw_s, v_jw_t\}$, or $\{v_jw_s, w_sw_t\}$ forms an incompatible pair in $(G,\mathcal{F})$.
\end{enumerate}
Note that for any such $u_i$ (or $v_j$) as in \ref{p5} (resp.~\ref{p6}), the number of pair $(w_s,w_t)$ is at most $2\mu n^2$.
Hence the number of copies of $P^k_{\ell}$ in $\mathcal{K}$ satisfying \ref{p5} or \ref{p6} is at most \[2k(2\mu n^2+2\mu n^2)n^{\ell-2}<cm^{\ell}\le|\mathcal{K}|,\] where the first inequality follows since $\mu\ll\eps,c,\frac{1}{k}$. Thus we can pick from $\mathcal{K}$ a compatible copy of $P^k_{\ell}$ with a base path $Q$ as desired. This completes the entire proof.
\oof

\section{Acknowledgement}
The authors would like to thank Jie Han  and Guanghui Wang for helpful discussions.

\bibliographystyle{abbrv}
\bibliography{ref}

\end{document}